\documentclass{amsart}

\title[Matching polytopes and the totally non-negative Grassmannian]{Matching polytopes, toric geometry, and the totally 
non-negative Grassmannian}

\author{Alexander Postnikov}
\address{Massachusetts Institute of Technology}
\email{apost@math.mit.edu}

\author{David E Speyer}
\address{Massachusetts Institute of Technology}
\email{speyer@math.mit.edu}

\author{Lauren Williams}
\address{Harvard University}
\email{lauren@math.harvard.edu}

\thanks{Alexander Postnikov was supported 
in part by NSF CAREER Award DMS-0504629.  
David Speyer was supported by a research fellowship from the Clay Mathematics Institute.
Lauren Williams was supported in part by the NSF}

\subjclass[2000]{Primary 05Exx; Secondary 20G20, 14Pxx}

\keywords{Total positivity, Grassmannian,
CW complexes, Birkhoff polytope, matching, matroid polytope, cluster algebra}

\usepackage{pstricks, pst-node, amssymb}

\psset{unit=1pt, arrowsize=4pt, linewidth=.7pt}
\psset{linecolor=blue}
\newgray{grayish}{.90}
\newrgbcolor{embgreen}{0 .5 0}

\def\vblack(#1, #2)#3{\cnode*[linecolor=black](#1, #2){3}{#3}}
\def\vwhite(#1,#2)#3{\cnode[linecolor=black,fillcolor=white,fillstyle=solid](#1,#2){3}{#3}}
\countdef\x=23
\countdef\y=24
\countdef\z=25
\countdef\t=26

\def\tbox(#1,#2)#3{
\x=#1 \y=#2
\multiply\x by 12
\multiply\y by 12
\z=\x \t=\y
\advance\z by 12
\advance\t by 12
\psline(\x,\y)(\x,\t)(\z,\t)(\z,\y)(\x,\y)
\advance\x by 6
\advance\y by 6
\rput(\x,\y){{\bf #3}}}

\font\co=lcircle10

\def\jr{\rotatedown{\smash{\raise2pt\hbox{\co \rlap{\rlap{\char'005} \char'007}}
               \raise6pt\hbox{\rlap{\vrule height6.5pt}}
                \raise2pt\hbox{\rlap{\hskip4pt \vrule
          height0.4pt depth0pt
                width7.7pt}}}}}
\def\textcross{\ \smash{\lower4pt\hbox{\rlap{\hskip4.15pt\vrule height14pt}}
                \raise2.8pt\hbox{\rlap{\hskip-3pt \vrule height.4pt depth0pt
                        width14.7pt}}}\hskip12.7pt}

\def\textelbow{\ \hskip.1pt\smash{\raise2.75pt%
                \hbox{\co \hskip 4.15pt\rlap{\rlap{\char'004} \char'006}
                \lower6.8pt\rlap{\vrule height3.5pt}
                \raise3.6pt\rlap{\vrule height3.5pt}}
                \raise2.8pt\hbox{%
                  \rlap{\hskip-7.15pt \vrule height.4pt depth0pt
width3.5pt}%
                  \rlap{\hskip4.05pt \vrule height.4pt depth0pt
width3.5pt}}}
                \hskip8.7pt}

\usepackage{amsmath, amsthm, amssymb, amsbsy}
\usepackage{amsfonts, latexsym, stmaryrd, amscd, xy}
\usepackage[mathscr]{eucal}
\usepackage{epsfig}
\theoremstyle{plain}

\newtheorem{theorem}{Theorem}[section]

\newtheorem{proposition}[theorem]{Proposition}
\newtheorem{lemma}[theorem]{Lemma}

\newtheorem{example}[theorem]{Example}
\newtheorem{corollary}[theorem]{Corollary}

\theoremstyle{definition}
\newtheorem{remark}[theorem]{Remark}

\newtheorem{definition}[theorem]{Definition}

\usepackage{amsfonts}
\usepackage{xy}
\usepackage{amssymb}
\usepackage{amsmath}
\def\R{\mathbb{R}}
\def\Z{\mathbb{Z}}

\def\C{\mathbb{C}}

\def\M{\mathcal{M}}

\def\P{\mathbb{P}}

\def\wind{\mathrm{wind}}

\def\weight{\mathrm{weight}}

\def\mytilde{\kern-.015in\hbox{\lower.03in\hbox{\~{}}}\kern-.01in}

\def\Faces{\mathrm{Faces}}

\def\Meas{\mathrm{Meas}}
\def\tMeas{\widetilde{M}eas}

\def\wind{\mathrm{wind}}

\def\O{\mathcal{O}}

\def\Faces{\mathrm{Faces}}
\def\Edges{\mathrm{Edges}}

\newcommand{\thmrefer}[1]{\renewcommand\thetheorem
  {\protect\ref{#1}}\addtocounter{theorem}{-1}}

\xyoption{all}
\CompileMatrices

\begin{document}

\begin{abstract}
In this paper we use toric geometry to investigate the topology of 
the totally non-negative part of the Grassmannian, denoted $(Gr_{k,n})_{\geq 0}$.
This is a cell complex whose cells $\Delta_G$ can be parameterized
in terms of the combinatorics of plane-bipartite graphs $G$.
To each cell $\Delta_G$ we associate a 
certain
polytope $P(G)$.  The polytopes $P(G)$ are  analogous to
the well-known Birkhoff polytopes, and we 
describe their face lattices
in terms of matchings and unions of matchings of $G$.  We also
demonstrate a close connection between the polytopes $P(G)$ and
matroid polytopes.
We use the data of $P(G)$ to define
an associated toric variety $X_G$.  We use our technology to prove that
the cell decomposition of $(Gr_{k,n})_{\geq 0}$ is a CW complex, and 
furthermore, that 
the Euler characteristic of the closure of each cell of $(Gr_{k,n})_{\geq 0}$ is $1$.
\end{abstract}

\maketitle

\setcounter{tocdepth}{1}
\tableofcontents

\section{Introduction}

The classical theory of total positivity concerns matrices in 
which all minors are non-negative.  While this theory was pioneered
by Gantmacher, Krein, and Schoenberg in the 1930's, the past decade
has seen a flurry of research in this area initiated by 
Lusztig \cite{Lusztig1, Lusztig2, Lusztig3}. 
Motivated by surprising connections
he discovered between his theory of canonical bases for quantum
groups and the theory of total positivity, Lusztig extended
this subject by introducing the totally non-negative 
variety $\mathbb{G}_{\geq 0}$ in an arbitrary reductive group $\mathbb{G}$ and the 
totally non-negative part $(\mathbb{G}/P)_{\geq 0}$ of a real flag variety 
$\mathbb{G}/P$.  

Recently Postnikov \cite{Postnikov} investigated the 
combinatorics of the totally non-negative part of a 
Grassmannian $(Gr_{k,n})_{\geq 0}$: he established a relationship
between $(Gr_{k,n})_{\geq 0}$ and certain planar bicolored graphs,
producing a combinatorially explicit cell decomposition
of $(Gr_{k,n})_{\geq 0}$.  To each such graph $G$ he constructed 
a parameterization $\Meas_G$ of a corresponding cell of 
$(Gr_{k,n})_{\geq 0}$ by $(\R_{> 0})^{\# \Faces(G)-1}$.  
In fact, this cell decomposition
is a special case of a cell 
decomposition of $(\mathbb{G}/P)_{\geq 0}$
which was conjectured by 
Lusztig and proved by Rietsch \cite{Rietsch1}, although that 
cell decomposition was described in quite different terms.
Other combinatorial aspects of $(Gr_{k,n})_{\geq 0}$, and more 
generally of $(\mathbb{G}/P)_{\geq 0}$, were investigated
by Marsh and Rietsch \cite{MR}, Rietsch \cite{Rietsch2},
 and the third author \cite{Williams1, Williams2}.


It is known that $(\mathbb{G}/P)_{\geq 0}$ is contractible \cite{Lusztig1}
and it is conjectured that $(\mathbb{G}/P)_{\geq 0}$ with 
its cell decomposition is a regular CW complex homeomorphic to a ball.  In \cite{Williams2},
the third author proved the combinatorial analogue of this conjecture, 
proving that the partially ordered set (poset) of cells of 
$(\mathbb{G}/P)_{\geq 0}$ is in fact the poset of cells of a regular CW complex homeomorphic to a ball.

In this paper we give an approach to this conjecture which uses toric geometry
to extend $\Meas_G$ to a 
map onto the closure of the corresponding cell of $(Gr_{k,n})_{\geq 0}$.
Specifically, given a {\it plane-bipartite} graph $G$, we construct a toric variety $X_G$ and a rational map  $m_G : X_G \to Gr_{k,n}$. We show that $m_G$ is well-defined on the totally non-negative part of $X_G$ and that its image is the closure of the corresponding cell of $(Gr_{k,n})_{\geq 0}$. 
The totally non-negative part of $X_G$ is homeomorphic to 
a certain polytope (the {\it moment polytope}) 
which we denote $P(G)$, so we can equally well think of this result as a parameterization of our cell by $P(G)$. 
The restriction of $m_G$ to the toric interior of the non-negative part of $X_G$ (equivalently, to the interior of $P(G)$) is $\Meas_G$.

Our technology proves that the cell decomposition of the totally non-negative part of the Grassmannian 
is in fact a CW complex.  While our map $m_G$ is 
well-defined on $(X_G)_{\geq 0}$ (which is a closed ball) and 
is a homeomorphism on the interior, in general $m_G$ is not a homeomorphism on the boundary of 
$(X_G)_{\geq 0}$; therefore this does not lead directly to a proof of the conjecture.
However, we do obtain more evidence that the conjecture is true: using Williams' result \cite{Williams2}
that the face poset of $(\mathbb{G}/P)_{\geq 0}$ is {\it Eulerian}, it follows that
the Euler characteristic of the closure of each cell of $(Gr_{k,n})_{\geq 0}$ is $1$.

The most elegant part of our story is how the combinatorics of the 
plane-bipartite graph $G$ reflects
the structure of the 
polytope $P(G)$ and hence the structure of $X_G$.  See Table 
\ref{PlabicTable} for some of these connections.  
The torus fixed points of $X_G$ correspond to 
{\it perfect orientations} of $G$, 
equivalently, to {\it almost perfect matchings} of $G$. 
The other faces of $X_G$ correspond to certain 
\emph{elementary subgraphs} of $G$, that is, to unions of almost perfect matchings 
of $G$. Every face of $X_{G}$ is of the form $X_{G'}$ for some plane-bipartite
graph $G'$ obtained by deleting some edges of $G$, and 
$m_{G}$ restricted to $X_{G'}$ is $m_{G'}$. 
It will follow from this that, for every face $Z$ of $X_G$, the interior of $Z$ is mapped to the 
interior of a cell of the totally non-negative Grassmannian with fibers that 
are simply affine spaces. We hope that this explicit description of the 
topology of the parameterization will be useful in studying the topology 
of $(Gr_{k,n})_{\geq 0}$.

\begin{table}[h]
\begin{tabular}{|p{7.5 cm}|p{4 cm}|}
\hline
{\bf Plane-Bipartite graph $G$} & {\bf Polytope $P(G)$} \\
\hline 
$\# \Faces(G) - 1$ & Dimension of $P(G)$ \\
\hline
Perfect orientations / almost perfect matchings &
  Vertices of $P(G)$ \\
\hline
Equivalence classes of edges & Facets of $P(G)$ \\
\hline
Lattice of elementary subgraphs & Lattice of faces of $P(G)$  \\
\hline
\end{tabular}
\vspace{.5cm}
\caption{How $G$ reflects $P(G)$}
\label{PlabicTable}
\end{table}

The structure of this paper is as follows.  In Section~\ref{review} we review the combinatorics
of plane-bipartite graphs and perfect orientations.  Next, in Section~\ref{Toric} we review toric varieties
and their non-negative parts, and prove a lemma which is key to our CW complex result. 
We then, in Section~\ref{Match}, introduce the polytopes which will give rise to the toric varieties of interest to us.
Using these polytopes, in Section~\ref{MatroidPolytope} we make the connection between our polytopes $P(G)$ and 
matroid polytopes and explain the relation of our results to problems arising in cluster algebras and tropical geometry.
 In Section \ref{CWComplex} we use these polytopes to prove that 
 the cell decomposition of $(Gr_{k,n})_{\geq 0}$ is in fact a CW complex. 
In Section \ref{faces} we analyze the combinatorics of our polytopes in greater detail, giving a combinatorial
description of the face lattice of $P(G)$ in terms of matchings and unions of matchings of $G$.
Finally, in Section \ref{Numerology},
we calculate $f$-vectors, Ehrhart series, volumes, and the degrees of the 
corresponding toric varieties for a few small plane-bipartite graphs.

\textsc{Acknowledgements:} We are grateful to Vic Reiner for pointing
out the similarity between our polytopes $P(G)$ and Birkhoff polytopes, and to 
Allen Knutson for many helpful conversations.

\section{The totally non-negative Grassmannian and plane-bipartite graphs}
\label{review}

In this section we review some material from \cite{Postnikov}.
We have slightly modified the notation from \cite{Postnikov}
to make it more convenient for the present paper.

Recall that the (real) Grassmannian $Gr_{k,n}$ is the space of all
$k$-dimensional subspaces of $\R^n$, for $0\leq k\leq n$.  An element of
$Gr_{k,n}$ can be viewed as a full-rank $k\times n$ matrix modulo left
multiplications by nonsingular $k\times k$ matrices.  In other words, two
$k\times n$ matrices represent the same point in $Gr_{k,n}$ if and only if they
can be obtained from each other by row operations.

Let $\binom{[n]}{k}$ be the set of all $k$-element subsets of $[n]:=\{1,\dots,n\}$.
For $I\in \binom{[n]}{k}$, let $\Delta_I(A)$
denote the maximal minor of a $k\times n$ matrix $A$ located in the column set $I$.
The map $A\mapsto (\Delta_I(A))$, where $I$ ranges over $\binom{[n]}{k}$,
induces the {\it Pl\"ucker embedding\/} $Gr_{k,n}\hookrightarrow \mathbb{RP}^{\binom{n}{k}-1}$.

\begin{definition} {\rm \cite[Section~3]{Postnikov}} \
The {\it totally non-negative Grassmannian} $(Gr_{k,n})_{\geq 0}$
is the subset of the real Grassmannian $Gr_{k,n}$
that can be represented by $k\times n$ matrices $A$ 
with all  maximal minors
$\Delta_I(A)$ non-negative.

For $\M\subseteq \binom{[n]}{k}$, 
the {\it positive Grassmann cell\/} $C_\mathcal{M}$ is 
the subset of the elements in $(Gr_{k,n})_{\geq 0}$ represented by all $k\times n$ matrices $A$ with 
the prescribed collection of maximal minors strictly positive $\Delta_I(A)>0$,
for $I\in \mathcal{M}$, and the remaining
maximal minors equal to zero $\Delta_J(A)=0$, for $J\not\in \mathcal{M}$.

A subset $\M\subseteq \binom{[n]}{k}$ such that $C_\mathcal{M}$ is nonempty 
satifies the base axioms of matroid.  These special matroids are called 
{\it positroids.}
\end{definition}

Clearly $(Gr_{k,n})_{\geq 0}$ is a disjoint union of the positive Grassmann cells $C_\mathcal{M}$.
It was shown in \cite{Postnikov} that each of these cells $C_\mathcal{M}$ is really a cell,
that is, it is homeomorphic to an open ball of appropriate dimension $d$.
Moreover, one can explicitly construct a parametrization $\R_{>0}^d \buildrel\sim\over\to  C_\mathcal{M}$ 
using certain planar graphs, as follows.

\begin{definition}
A {\it plane-bipartite graph\/} is an undirected graph $G$ drawn inside a disk
(considered modulo homotopy) 
with $n$ {\it boundary vertices\/} on the boundary of the disk,
labeled $b_1,\dots,b_n$ in clockwise order, as well as some 
colored {\it internal vertices\/}.   
These  internal vertices 
are strictly inside the disk and are 
colored in black and white such that:
\begin{enumerate}
\item Each edge in $G$ joins two vertices of different colors.
\item Each boundary vertex $b_i$ in $G$ is incident to a single edge.
\end{enumerate}

A {\it perfect orientation\/} $\O$ of a plane-bipartite graph $G$ is a 
choice of directions of its edges such that each 
black internal vertex $u$ is incident to exactly one edge
directed away from $u$; and each white internal vertex $v$ is incident
to exactly one edge directed towards $v$.  
A plane-bipartite graph is called {\it perfectly orientable\/} if it has a perfect orientation.  
Let $G_\O$ denote the directed graph associated with a perfect orientation $\O$ of $G$. The {\it source set\/} $I_\O \subset [n]$ of a perfect orientation $\O$ is the set of $i$ for which $b_i$ 
is a source of the directed graph $G_\O$. Similarly, if $j \in \bar{I}_{\O} := [n] \setminus I_{\O}$, then $b_j$ is a sink of $\O$.
\end{definition}

All perfect orientations of a fixed
$G$ have source sets of the same size $k$ where
$k-(n-k) = \sum \mathrm{color}(v)\,(\deg(v)-2)$. Here the sum is over all internal vertices $v$, $\mathrm{color}(v) = 1$ for a black vertex $v$, and $\mathrm{color}(v) = -1$ for a white vertex;  
see~\cite{Postnikov}.  In this case we say that $G$ is of {\it type\/} $(k,n)$.

Let us associate a variable $x_e$ with each edge of $G$.
Pick a perfect orientation $\O$ of $G$.  For $i\in I_\O$ and $j\in \bar I_\O$, 
define the {\it boundary measurement\/} $M_{ij}$ as the following 
power series in the $x_e^{\pm 1}$:
$$
M_{ij}:=\sum_{P} (-1)^{\wind(P)}\, x^P,
$$
where the sum is over all directed paths in $G_\O$ that start at the boundary vertex $b_i$
and end at the boundary vertex $b_j$.  The Laurent monomial $x^P$ is given by $x^P:=\prod' x_{e'}/\prod'' x_{e''}$,
where the product $\prod'$ is over all edges $e'$ in $P$ directed 
from a white vertex to a black vertex,
and the product $\prod''$ is over all edges $e''$ in $P$ directed 
from a black vertex to a white vertex.
For any path $P$, let $\sigma_1$,$\sigma_2$, \ldots, $\sigma_r \in \R/2 \pi \Z$ be the directions of the edges of $P$ (in order). Let $Q$ be the path through $\R/2 \pi \Z$ which travels from $\sigma_1$ to $\sigma_2$ to $\sigma_3$ and so forth, traveling less then $\pi$ units of arc from each $\sigma_i$ to the next. The {\it winding index\/} $\wind(P)$ is the number of times $Q$ winds around the circle $\R/2 \pi \Z$, rounded to the nearest integer.
The index $\wind(P)$ is congruent to the number of self-intersections of the path $P$
modulo $2$.

\begin{remark}
Let us mention several differences in the notations given above 
and the ones from \cite{Postnikov}.
The construction in \cite{Postnikov} was done for {\it plabic graphs,} which
are slightly more general than the plane-bipartite graphs defined above.   Edges
in plabic graphs are allowed to join vertices of the same color.
One can easily transform a plabic graph into a plane-biparte graph, without
much change in the construction, by contracting edges
which join vertices of the same color, or alternatively, by inserting vertices
of different color in the middle of such edges.

Another difference is that we inverted the edge variables 
from \cite{Postnikov} for all edges directed from a black
vertex to a white vertex.

In \cite{Postnikov} the boundary measurements $M_{ij}$ were
defined for any planar directed graph drawn inside a disk.  
It was shown that one can easily transform 
any such graph into a plane-bipartite graph with a perfect 
orientation of edges that has the same boundary measurements.
\end{remark}

Let $E(G)$ denote the edge set of a plane-bipartite graph $G$, and let $\R_{>0}^{E(G)}$ denote the set of 
vectors $(x_e)_{e\in E(G)}$ with strictly positive real coordinates $x_e$.

\begin{lemma}
\label{lem:Mij_subtractive_free}
\cite[Lemma 4.3]{Postnikov}
The sum in each boundary measurement $M_{ij}$ evaluates to a subtraction-free rational expression 
in the $x_e$.  Thus it gives a well-defined positive function on $\R_{> 0}^{E(G)}$. 
\end{lemma} 

For example, suppose that $G$ has two boundary vertices, $1$ and $2$ and two internal vertices $u$ and $v$, with edges $a$, $b$, $c$ and $d$ running connecting $1 \to u$, $u \to v$, $v \to u$ and $v \to 2$. Then $M_{12} = abd - abcbd + abcbcbd - \cdots = abd/(1+bc)$. The sum only converges when $|bc| < 1$ but, by interpreting it as a rational function, we can see that it gives a well defined value for any $4$-tuple $(a,b,c,d)$ of positive reals.

If the graph $G_\O$ is acyclic then there are finitely many directed paths $P$,
and $\wind(P)=0$ for any $P$.  In this case the $M_{ij}$ are clearly Laurent polynomials in the $x_e$
with positive integer coefficients, and the above lemma is trivial.

For a plane-biparte graph $G$ of type $(k,n)$ and a perfect orientation $\O$ with the source set $I_\O$, 
let us construct the $k\times n$ matrix $A=A(G,\O)$ such that
\begin{enumerate}
\item The $k\times k$ submatrix of $A$ in the column set $I_\O$ is the identity matrix.
\item  For any $i\in I_\O$ and $j\in \bar I_\O$, the minor $\Delta_{(I_\O\setminus \{i\})\cup \{j\}}(A)$ equals 
$M_{ij}$.
\end{enumerate}
These conditions uniquely define the matrix $A$.  Its entries outside the column set $I_\O$ are
$\pm M_{ij}$.  The matrix $A$ represents an element of the Grassmannian $Gr_{k,n}$.
Thus, by Lemma~\ref{lem:Mij_subtractive_free},
it gives the well-defined {\it boundary measurement map\/}
$$
\Meas_G:\R_{>0}^{E(G)}\to Gr_{k,n}.
$$

Clearly, the matrix $A(G,\O)$ described above will be different for different
perfect orientations $\O$ of $G$.  However, all these different matrices
$A(G,\O)$ represent the same point in the Grassmannian $Gr_{k,n}$.

Note that once we have constructed the matrix $A$, we can determine
which cell of $(Gr_{k,n})_{\geq 0}$ we are in by simply noting
which maximal minors are nonzero and which are zero.

\begin{proposition} 
\label{prop:same}
\cite[Theorem 10.1]{Postnikov}
For a perfectly orientable plane-bipartite graph $G$, the boundary measurement map
$\Meas_G$ does not depend on a choice of perfect orientation of $G$.
\end{proposition}

If we multiply the edge variables $x_e$ for all edges incident to an internal
vertex $v$ by the same factor, then the boundary measurement $M_{ij}$ will not
change.  Let $V(G)$ denote the set of internal vertices of $G$.  Let
$\R_{>0}^{E(G)/V(G)}$ be the quotient of $\R_{>0}^{E(G)}$ modulo the action of
$\R_{>0}^{V(G)}$ given by these rescalings of the $x_e$.  If the graph $G$ does
not have isolated connected components without boundary vertices\footnote{Clearly, we
can remove all such isolated components without affecting the boundary measurements.}, then
$\R_{>0}^{E(G)/V(G)} \simeq \R_{>0}^{|E(G)|-|V(G)|}$.
The map $\Meas_G$ induces the map
$$
\tMeas_G: \R_{>0}^{E(G)/V(G)} \to Gr_{k,n},
$$
which (slightly abusing the notation) we also call the boundary measurement map.

Talaska \cite{Talaska} has given an explicit combinatorial formula for the 
maximal minors (also called Pl\"ucker coordinates) of such matrices $A=A(G,\O)$.  
To state her result, we need a few definitions.
A {\it conservative flow} in a perfect orientation $\O$ of $G$ is a (possibly empty)
collection of pairwise vertex-disjoint oriented cycles.  (Each cycle is self-avoiding,
i.e. it is not allowed to pass through a vertex more than once.)  For 
$|J|=|I_{\O}|$, a {\it flow from $I_{\O}$ to $J$} is a collection of self-avoiding walks
and cycles, all pairwise vertex-disjoint, such that the sources of these walks are $I_{\O} \setminus (I_{\O} \cap J)$
and the destinations are $J \setminus (I_{\O} \cap J)$.  So a conservative flow can also be described as a flow from $I_{\O}$ to $I_{\O}$. The {\it weight} 
$\weight(F)$ of a flow $F$ is the product
of the weights of all its edges directed from the white to the black vertex,
divided by the product of all its edges directed from the black to the white 
vertex.\footnote{Note that here we slightly differ from Talaska's convention in order
to be consistent with our previous convention in defining $M_{ij}$.}
A flow with no edges has weight $1$.

\begin{theorem} \cite[Theorem 1.1]{Talaska}\label{TalaskaTheorem}
\label{prop:noncrossing}
Fix a perfectly orientable $G$ and a perfect orientation $\O$.
The minor $\Delta_J(A)$ of $A=A(G,\O)$, with columns in position $J$, is given by 
$$
\Delta_J = \left(\sum_{F} \weight(F)\right)/\left(\sum_{F'} \weight(F')\right).
$$
Here the sum in the numerator is over flows $F$
from $I_{\O}$ to $J$ and the sum in the denominator is over all 
conservative flows $F'$.
\end{theorem}

A point in the Grassmannian only depends on its Pl\"ucker coordinates up to multiplication by a common scalar. 
For our purposes, it is best to clear the denominators in Theorem~\ref{TalaskaTheorem}, and give a purely (Laurent)
polynomial formula:
\begin{corollary} \label{TalaskaCor}
Using the notation of Theorem~\ref{TalaskaTheorem}, the point of $Gr_{k,n}$ corresponding to the row span of $A$ has Pl\"ucker coordinates
$$p_J :=  \left(\sum_{F} \weight(F)\right)$$
where the sum is over flows $F$ from $I_{\O}$ to $J$.
\end{corollary}

Theorem~\ref{prop:noncrossing} implies that the image of the boundary 
measurement map $\tMeas_G$ lies in the totally non-negative Grassmannian
$(Gr_{k,n})_{\geq 0}$.  Moreover, the image is equal to a certain 
positive cell in $(Gr_{k,n})_{\geq 0}$.

\begin{proposition}
\label{prop:plabic_cells}
\cite[Theorem~12.7]{Postnikov}
Let $G$ be any perfectly orientable plane-bipartite graph of type $(k,n)$.
Then the image of the boundary measurement map $\tMeas_G$ is a certain positive Grassmann cell 
$C_\mathcal{M}$ in $(Gr_{k,n})_{\geq 0}$. 
For every cell $C_\mathcal{M}$ in $(Gr_{k,n})_{\geq 0}$, there is a perfectly orientable plane-bipartie graph $G$ 
such that $C_\mathcal{M}$ is the image of $\tMeas_G$. 
The map $\tMeas_G$ is a fiber bundle with fiber an 
$r$-dimensional affine space, for some non-negative $r$.
For any cell of $(Gr_{k,n})_{\geq 0}$, we can always choose a graph $G$ such that $\tMeas_G$ is a homeomorphism onto this cell. 
\end{proposition}

Let us say that a plane-bipartite graph $G$ 
is {\it reduced\/} if $\tMeas_G$ is a homeomorphism,
and $G$ has no isolated connected components
nor internal vertices incident to a single edge;
see \cite{Postnikov}.

An {\it almost perfect matching} of a plane-bipartite graph $G$ is a 
subset $M$ of edges such that each internal vertex is incident to 
exactly one edge in $M$ (and the boundary vertices $b_i$ are incident
to either one or no edges in $M$). 
There is a bijection between perfect orientations of $G$ and almost perfect
matchings of $G$ where, for a perfect orientation $\O$ of $G$, an edge $e$ is
included in the corresponding matching if $e$ is directed away 
from a black vertex
or to a white vertex in $\O$.\footnote{Note that typically 
$e$ is directed away from a black vertex if and only if it is 
directed towards a white vertex.  However, we have used the 
word {\it or} to make the bijection well-defined when
boundary vertices are not colored.}

For a plane-bipartite graph $G$ and the corresponding cell
$C_\mathcal{M} = \mathrm{Image}(\Meas_G)$ in $(Gr_{k,n})_{\geq 0}$,
one can combinatorially construct the matroid $\mathcal{M}$
from the graph $G$, as follows.

\begin{proposition}
\cite[Propostion~11.7, Lemma~11.10]{Postnikov}
A subset $I\in \binom{[n]}{k}$ is a base of the matroid $\M$ if and only 
there exists a perfect orientation $\O$ of $G$ such that $I=I_\O$.

Equivalently, assuming that all boundary vertices $b_i$ in $G$ are black,
$I$ is a base of $\M$ if and only if there exists an almost perfect
matching $M$ of $G$ such that 
$$
I = \{i\mid b_i \text{ belongs to an edge from $M$}\}. 
$$
\end{proposition}



\section{Toric varieties and their non-negative parts}
\label{Toric}

We may define a (generalized) projective toric variety
as follows \cite{Cox, Sottile}.
Let $S=\{\mathbf{m}_i \ \vert \ i=1, \dots, \ell\}$ be any finite subset
of $\Z^n$, where $\Z^n$ can be thought of as the character group
of the torus $(\C^*)^n$.
Here
$\mathbf{m}_i=(m_{i1}, m_{i2},\dots ,m_{in})$.
Then consider
the map $\phi: (\C^*)^n \to \P^{\ell-1}$ such that
$\mathbf{x}=(x_1, \dots , x_n) \mapsto [\mathbf{x^{m_1}}, \dots , \mathbf{x^{m_\ell}}]$.
Here $\mathbf{x^{m_i}}$ denotes $x_1^{m_{i1}} x_2^{m_{i2}} \dots x_n^{m_{in}}$.
We then define the toric variety $X_S$
to be the Zariski closure of the image of this map. We write $\tilde{\phi}$ for the inclusion of $X_S$ into $\P^{\ell-1}$
The {\it real part} $X_S(\R)$ of $X_S$ is defined to be the
intersection of $X_S$ with $\R\P^{\ell-1}$; the
{\it positive part} $X_S^{>0}$ is defined to be the image of
$(\R_{>0})^n$ under $\phi$; and the {\it non-negative part}
$X_S^{\geq 0}$ is defined to be the closure of
$X_S^{>0}$ in $X_S(\R)$. We note for future reference that $X_S$, $X_S(\R)$ and $X_S^{\geq 0}$ are unaltered by translating the set $S$ by any integer vector.

Note that $X_S$ is not necessarily a toric variety in the sense of \cite{Fulton}, as
it may not be normal;
however, its normalization is a toric variety in that sense.  See \cite{Cox} for more details.

Let $P$ be the convex hull of $S$.
There is a homeomorphism from $X_S^{\geq 0}$ to $P$, known as the moment map.
(See \cite[Section 4.2, page 81]{Fulton}
and \cite[Theorem 8.4]{Sottile}).
In particular,
$X_S^{\geq 0}$ is homeomorphic to a closed ball.

%

We now prove a simple but very important lemma.

\begin{lemma}\label{important}
Suppose we have a map $\Phi: (\R_{>0})^n \to \P^{N-1}$ given by
\begin{equation*}
(t_1, \dots , t_n) \mapsto [h_1(t_1,\dots,t_n), \dots , h_N(t_1,\dots,t_n)],
\end{equation*}
where the $h_i$'s are Laurent polynomials with positive coefficients.  Let $S$ be the
set of all exponent vectors in $\Z^n$ which occur among the (Laurent) monomials
of the $h_i$'s, and let $P$ be the convex hull of the points of $S$.
Then the map $\Phi$ factors through the totally positive part
$(X_P)_{>0}$, giving a map
$\tau_{>0}: (X_P)_{>0} \to \P^{N-1}$.  Moreover $\tau_{>0}$ extends continuously to the
closure to give a well-defined map
$\tau_{\ge 0}:(X_P)_{\ge 0} \to \overline{\tau_{>0}((X_{P})_{>0})}$.
\end{lemma}

\begin{proof}
Let $S = \{\mathbf{m_1},\dots,\mathbf{m_{\ell}}\}$.
Clearly the map $\Phi$ factors
as the composite map $t=(t_1,\dots,t_n) \mapsto
 [\mathbf{t^{m_1}}, \dots , \mathbf{t^{m_\ell}}] \mapsto [h_1(t_1,\dots,t_n),\dots,
       h_N(t_1,\dots,t_n)]$,
and the image of $(\R_{>0})^n$ under the first map is precisely
$(X_P)_{>0}$.
The second map, which we will call $\tau_{>0}$,
takes a point $[x_1,\dots, x_{\ell}]$ of $(X_P)_{>0}$ to
   $[g_1(x_1,\dots,x_{\ell}), \dots, g_N(x_1,\dots,x_{\ell})]$,
where the $g_i$'s are homogeneous polynomials of degree $1$ with positive coefficients.
By construction, each $x_i$ occurs in at least one of the $g_i$'s.

Since $(X_P)_{\geq 0}$ is the closure inside $X_P$ of $(X_P)_{>0}$,
any point $[x_1,\dots,x_{\ell}]$ of $(X_P)_{\geq 0}$ has all $x_i$'s non-negative;
furthermore, not all of the $x_i$'s are equal to $0$.  And now since the $g_i$'s
have positive coefficients and they involve {\it all} of the $x_i$'s, the image of
any point $[x_1,\dots,x_{\ell}]$ of $(X_P)_{\geq 0}$ under $\tau_{>0}$ is well-defined.
Therefore $\tau_{>0}$ extends continuously to the closure to give a well-defined map
$\tau_{\ge 0}:(X_P)_{\ge 0} \to \overline{\tau_{>0}((X_{P})_{>0})}$.

\end{proof}

In  Section \ref{CWComplex}
 we will use this lemma to prove that $(Gr_{k,n})_{\geq 0}$ 
is a 
CW complex.

\section{Matching polytopes for plane-bipartite graphs} \label{Match}

In this section we will define a family of polytopes $P(G)$ associated to
plane-bipartite graphs $G$.

\begin{definition}
Given an almost perfect 
matching of a plane-bipartite graph 
$G$, we associate to it the 0-1 vector in $\R^{E(G)}$ 
where the coordinates associated to edges in the matching are $1$ and all
other coordinates are $0$.
We define $P(G)$ to be the convex hull of these $0$-$1$ vectors.
\end{definition}

\begin{remark}
Note that more generally, we could define $P(G)$ for any 
graph $G$ with a distinguished subset of ``boundary" vertices.
Many of our forthcoming results about $P(G)$ for plane-bipartite
graphs $G$ should be extendable to this generality.
\end{remark}

Because all of the $0$-$1$ vectors above have the property that 
$\sum_{e \ni v} x_e =1$ for all internal vertices $v$ of $V(G)$, the polytope 
$P(G)$ lies in the subspace of 
$\R^{E(G)}$ defined by $\{\sum_{e \ni v} x_e =1 \ \vert \ v\in V(G) \}$.

We will now see how one can arrive at these polytopes in another way.
Recall that for each $G$ we have the boundary measurement map
$\tMeas_G: \R_{>0}^{E(G)/V(G)} \to Gr_{k,n}$.  Embedding the image
into projective space via the Pl\"ucker embedding, we have 
an explicit formula for the coordinates given by Talaska (Corollary
 \ref{TalaskaCor}).

In the following definition, we use the notation of Theorem 
\ref{TalaskaTheorem}.

\begin{definition}\label{polytopedef}
Fix a perfect orientation $\O$ of $G$.
We define $P(G,\O)$ to be the convex hull 
of the exponent vectors of the weights of all flows starting
at $I_{\O}$.  A priori this polytope lies
in $\R^{E(G)}$, but we will see that 
$P(G,\O)$ lies in a subspace of 
$\R^{E(G)}$.
\end{definition}

\begin{remark}\label{usefulforCW}
Note that what we are doing in Definition \ref{polytopedef} 
is taking the convex hull of all exponent vectors 
which occur in the $p_J(A)$ from 
Corollary~\ref{TalaskaCor}, as $J$ ranges over all subsets of columns
of size $|I_{\O}|$.
\end{remark}

We now relate $P(G)$ and $P(G, \O)$. We continue to use the notion of flows introduced in shortly before Theorem~\ref{TalaskaTheorem}.

\begin{lemma}\label{differ}
Fix a plane-bipartite graph $G$ and a perfect orientation $\O_1$.
If we choose a flow in $\O_1$ and switch the direction of all edges
in this flow, we obtain another perfect orientation.  Conversely, 
one can obtain any perfect orientation $\O_2$ of $G$
from $\O_1$ by switching all directions
of edges in a flow
in $\O_1$.  
\end{lemma}

\begin{proof}
The first claim is simple: a perfect orientation is one in which
each black vertex has a unique outcoming edge and each white vertex
has a unique incoming edge.  If we switch the orientation of all edges
along one of the paths or cycles in the flow, clearly this property will
be preserved.

To see the converse, let $E'$ denote the set of edges of $G$ in which 
the orientations $\O_1$ and $\O_2$ disagree.  
It follows from the definition of perfect orientation that
every edge $e$ in $E'$ incident
to some vertex $v$ can be paired uniquely with another edge $e'$ in $E'$
which is also incident to $v$ (note that at each vertex $v$ of $G$ there
are either $0$ or $2$ incident edges which are in $E'$).  This pairing
induces a decomposition of $E'$ into a union of 
vertex-disjoint 
(undirected) 
cycles and paths.  Moreover, each such cycle or
path is directed in both $\O_1$ and $\O_2$ (but of course in
opposite directions).
This set of cycles and paths is the relevant flow.
\end{proof}

Because of the bijection between perfect orientations and almost perfect matchings
(see Section \ref{review}), Lemma \ref{differ} implies the following.

\begin{corollary}
Fix $G$ and a perfect orientation $\O$.  Flows in $\O$ are in bijection with 
perfect orientations of $G$ (obtained by reversing all edges of the flow in $\O$)
which are in bijection with almost perfect matchings
of $G$.
\end{corollary}

We can now see the following.

\begin{corollary}
For any perfect orientation $\O$, the polytope $P(G,\O)$ is a translation of $P(G)$
by an integer vector.
\end{corollary}

\begin{proof}
Let $F$ denote the empty flow on $\O$, 
$F'$ be some other flow in $\O$, and $\O'$ the perfect orientation obtained from $\O$
by reversing the directions of all edges in $F'$.  Let $M$ and $M'$ be the almost
perfect matchings associated to $\O$ and $\O'$.  Let $x(F)$, $x(F')$, 
$x(M)$, and $x(M')$ be the vectors in $\R^{E(G)}$ associated to this flow
and these perfect orientations.  Of course $x(F)$ is the all-zero vector.
We claim that $x(M')-x(M)=x(F)-x(F')$.  

Fix an edge $e$ of $G$: we will check that the $e$-coordinates of $x(M')-x(M)$ and $x(F)-x(F')$
are equal. First, suppose that $e$ does not occur in $F'$. Then either $e$ appears in both $M$ and $M'$, or in neither. 
So $x(F)_e=x(F')_e=0$ and either $x(M)_e=x(M')_e=0$ or $x(M)_e=x(M')_e=1$. 
Now, suppose that $e$ occurs in $F'$, and is oriented from its white to its black endpoint in $\O$.
So $x(F)_e=0$ and $x(F')=1$.
The edge $e$ occurs in the matching $M'$ and not in the matching $M$, so $x(M)_e=0$ and $x(M')_e=1$.
Finally, suppose $e$ occurs in $F'$, and is oriented from its black to its white endpoint in $\O$.
Then $x(F)_e=0$ and $x(F')=-1$. 
The edge $e$ occurs in the matching $M$ and not in the matching $M'$, so $x(M)_e=1$ and $x(M')_e=0$.
\end{proof}

In particular, up to translation, $P(G,\O)$ does not depend on $\O$. Recall that translating a polytope does not affect the corresponding toric variety. 

%

In Figure \ref{P2}, we fix a plane-bipartite graph $G$ corresponding to the cell 
of $(Gr_{2,4})_{\geq 0}$ such that the Pl\"ucker coordinates 
$P_{12}, P_{13}, P_{14}$ are positive and all others are $0$.
We display the three perfect orientations
and the vertices of  $P(G)$.  

\begin{figure}[h]
\centering
\includegraphics[height=1.5in]{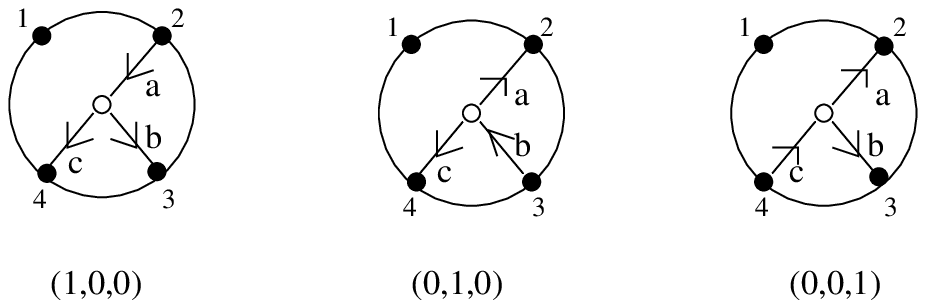}
\caption{}
\label{P2}
\end{figure}

In Figure \ref{P1P1}, we fix a plane-bipartite graph $G$ corresponding to the cell 
of $(Gr_{2,4})_{\geq 0}$ such that the Pl\"ucker coordinates 
$P_{12}, P_{13}, P_{24}, P_{34}$ are positive while 
$P_{14}$ and $P_{23}$ are $0$.
We display the four perfect orientations
and the vertices of $P(G)$.

\begin{figure}[h]
\centering
\includegraphics[height=1.5in]{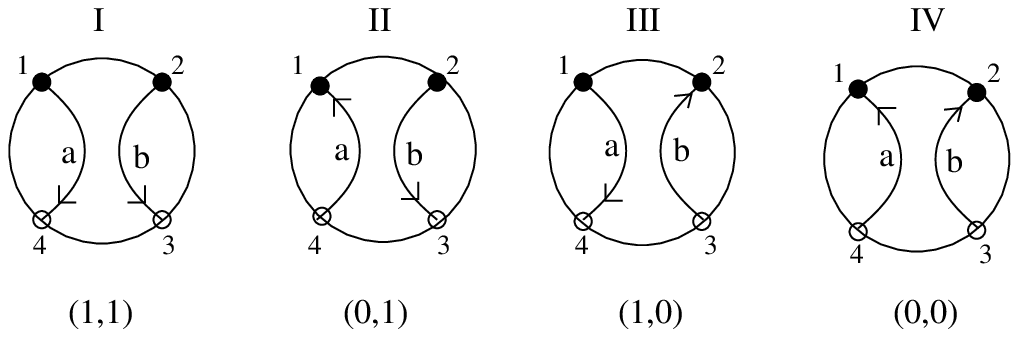}
\caption{}
\label{P1P1}
\end{figure}



In Figure \ref{Graph} we have 
fixed a plane-bipartite graph $G$ corresponding to the 
top-dimensional cell of $(Gr_{2,4})_{\geq 0}$.
$G$ has seven perfect orientations.
We have
drawn the edge graph of the
four-dimensional polytope $P(G)$.
This time
we have depicted the vertices of $P(G)$ using matchings instead of 
perfect orientations.  Next to each matching, we have also 
listed the source set of the corresponding perfect orientation.

\begin{figure}[h]
\centering
\includegraphics[height=3.2in]{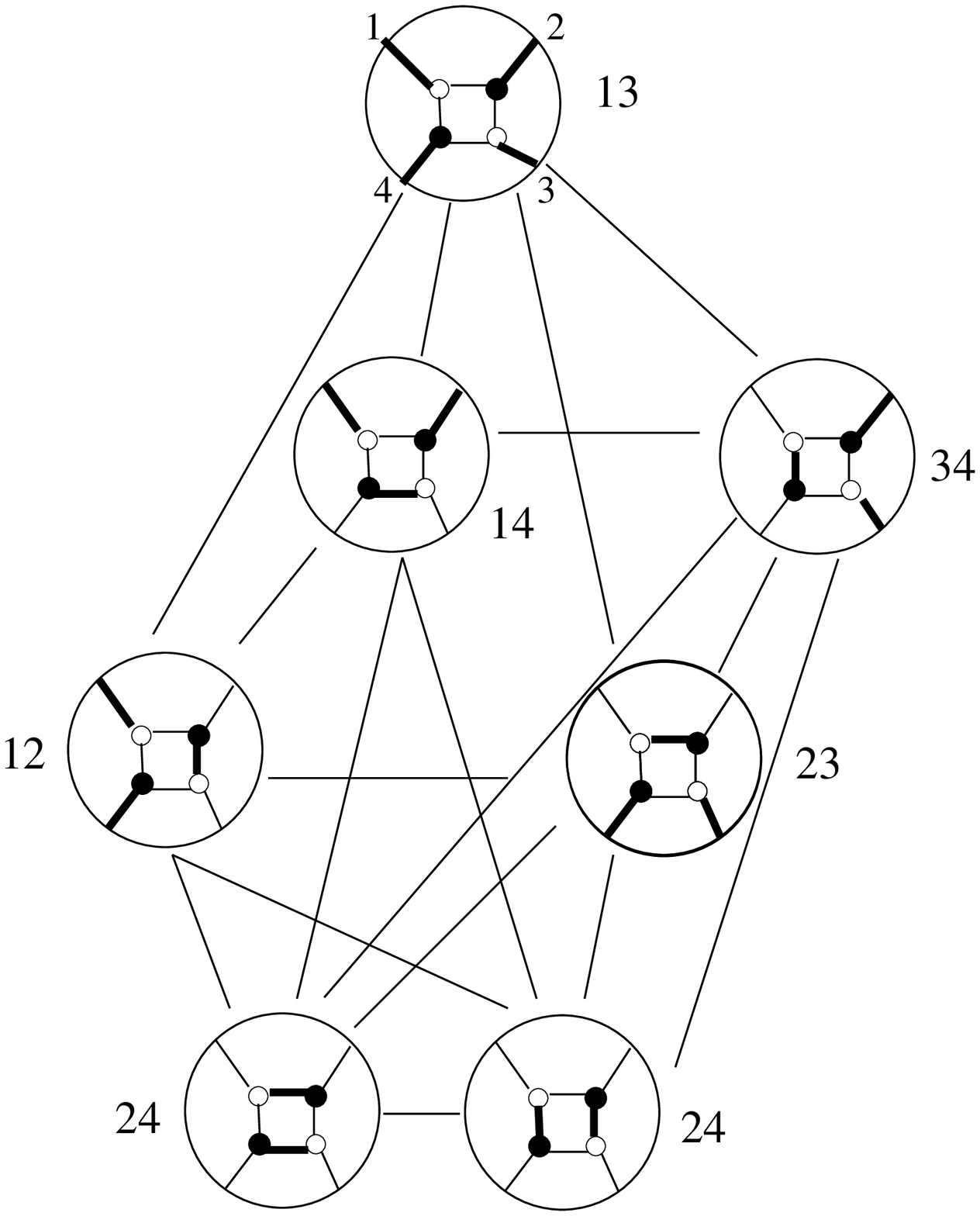}
\caption{}
\label{Graph}
\end{figure}

 \section{Connections with matroid polytopes and cluster algebras}\label{MatroidPolytope} 

 Every perfectly orientable plane-bipartite graph 
 encodes a realizable {\it positroid}, that is, an oriented matroid in which 
 all orientations are positive.  The bases of the positroid  associated to a plane-bipartite
 graph $G$ of type $(k,n)$ are precisely the $k$-element subsets $I \subset [n]$ which 
 occur as source sets of perfect orientations of $G$.  This is easy to see, 
as each 
 perfect orientation of $G$ gives rise to a parametrization of the cell $\Delta_G$ of 
 $(Gr_{k,n})_{\geq 0}$ in which the Pl\"ucker coordinate corresponding to 
the source set $I$ is $1$.
 Furthermore, if one takes a 
 (directed) path in a perfect orientation $\O$ and switches the orientation of each of its edges, this encodes a basis exchange.
 
 Given this close connection of perfectly orientable plane-bipartite graphs to positroids, it is 
 natural to ask whether there is a connection between our polytopes $P(G)$ and matroid 
 polytopes.  We first recall the definition of a matroid polytope.
 Let $M$ be a matroid of rank $k$ on the ground set $[n]$.  The 
 {\it matroid polytope} $Q(M)$ is the convex hull of the vectors 
 $\{ e(J) \ \vert \ J \text{ is a basis of }M\}$ where $e(J)$ is the 
 $0-1$ vector in $\R^n$ whose $i$th coordinate is $1$ if $i \in J$ and 
 is $0$ otherwise \cite{GGMS}.  The vertices are in one-to-one correspondence with 
 bases of $M$.  This polytope lies in the hyperplane
 $x_1 + \dots + x_n = 0$ and, if the matroid $M$ is connected, has dimension $n-1$. 

\begin{proposition}  \label{projection}
There is a linear projection $\Psi$ from $P(G)$ to 
$Q(M_G)$.  
The fibers of this projection over the vertices of $Q(M_G)$ are the 
Newton polytopes for the Laurent polynomials which express the 
Pl\"ucker coordinates on $X_G$ in terms of the edge variables. 
\end{proposition}

\begin{proof} 
If $G$ is a plane-bipartite graph of type $(k,n)$, one can associate 
to each vertex $v_{M}$ of $P(G)$ the basis of the corresponding positroid 
corresponding to the boundary edges which are matched in $G$. In terms of the 
bijection between perfect matchings and perfect orientations, this is the source 
set of the corresponding perfect orientation.  This gives the linear projection 
$\Psi$ from $P(G)$ to $Q(M_G)$.  To see that the statement about the fibers
is true, 
see Corollary~\ref{TalaskaCor}, and remember the relationship 
between matchings and flows.
\end{proof}

The second and third authors, in 
\cite{SpeyerWilliams}, related the Newton polytopes of Proposition \ref{projection}
to the 
positive part of the tropical Grassmannian; our results in that 
paper can be summarized by saying that the positive part of the 
tropical Grassmannian is combinatorially isomorphic to the dual 
fan of the fiber polytope of the map $P(G) \to Q(M_G)$. 
\footnote{We worked with {\it face variables} rather than edge 
variables in \cite{SpeyerWilliams}, but the two corresponding
realizations of $P(G)$ are linearly isomorphic.}

The fact that the Pl\"ucker coordinates on $X_G$ can all be expressed as Laurent polynomials in the edge weights is not simply a fortunate coincidence, but is a consequence\footnote{This consequence is not completely straightforward; one must express certain ratios of the edge weights as Laurent monomials in the variables of a certain cluster, and this involves a nontrivial ``chamber Ansatz''.} of the fact that the coordinate ring of $X_G$ has the structure of a cluster algebra. (See \cite{FZ} for the definition of cluster algebras, \cite{Scott} for the verification that the largest cell of the Grassmannian has the structure of a cluster algebra and \cite{Postnikov} for the fact that every $X_G$ has this structure.) 
In general, if we had a better understanding of the Newton polytopes of Laurent polynomials arising from cluster algebras, we could resolve many of 
the open questions in that theory.

 \begin{example}
 Consider the plane-bipartite graph $G$ from Figure \ref{Graph}.  
 This corresponds
 to the positroid of rank two on the ground set $[4]$ such that 
 all subsets of size $2$ are independent.  The edge graph of 
 the four-dimensional polytope $P(G)$ is shown in Figure \ref{Graph},
 and each vertex is labeled with the basis it corresponds to.
 The matroid polytope of this matroid is the 
 (three-dimensional) octahedron with 
 six vertices corresponding to the two-element subsets of $[4]$.
 Under the map $\Psi$, each vertex of $P(G)$ corresponding to 
 the two-element subset $ij$ gets mapped to the vertex of the 
 octahedron whose $i$th and $j$th coordinates are $1$ (all 
 other coordinates being $0$).
 \end{example}

\section{$(Gr_{k,n})_{\geq 0}$ is a CW complex}\label{CWComplex}

We  now prove that the cell decomposition of $(Gr_{k,n})_{\geq 0}$
is a CW complex, and obtain as a corollary that the Euler characteristic
of the closure of each cell is $1$.

To review the terminology, a {\it cell complex} is 
a decomposition of a space $X$ into a disjoint union
of {\it cells}, that is 
open balls.  A {\it CW complex} is a cell complex together 
with the extra data of {\it attaching maps}.  
More specifically, each cell in a CW complex is {\it attached} by 
gluing a closed $i$-dimensional ball $D^i$ to the $(i-1)$-skeleton
$X_{i-1}$, i.e.\ the union of all lower dimensional cells.
The gluing is specified by a continuous function $f$ from 
$\partial D^i = S^{i-1}$ to $X_{i-1}$.  CW complexes are defined 
inductively as follows: Given $X_0$ a discrete space
(a discrete union of $0$-cells), and inductively constructed subspaces
$X_i$ obtained from $X_{i-1}$ by attaching some collection of $i$-cells,
the resulting colimit space $X$ is called a {\it CW complex} provided
it is given the weak topology and every closed cell is covered by a finite 
union of open cells.

Although we don't need this definition here, we note that a {\it regular}
CW complex is a CW complex such that the closure of each cell 
is homeomorphic to a closed ball and the boundary of each cell is 
homeomorphic to a sphere.  It is not known if 
the cell decomposition of $(Gr_{k,n})_{\geq 0}$ is regular, although 
the results of \cite{Williams2} suggest that the answer is yes.

To prove our main result, we will also use the following lemma, which can be
found in \cite{Postnikov, Rietsch2}. 

\begin{lemma}\cite[Theorem 18.3]{Postnikov}, \cite[Proposition 7.2]{Rietsch2}\label{Closure}
The closure of a cell $\Delta$ in $(Gr_{k,n})_{\geq 0}$ is the union of 
$\Delta$ together with lower-dimensional cells.
\end{lemma}

\begin{theorem}
The cell decomposition of $(Gr_{k,n})_{\geq 0}$ is a finite CW complex.
\end{theorem}

\begin{proof}
All of these cell complexes contain only finitely many cells; therefore
the closure-finite condition in the definition of a CW complex
is automatically satisfied.
What we need to do is define the attaching maps for the cells: 
we need to prove that for each $i$-dimensional cell
there is a continuous map 
$f$ from $D^i$ to $X_{i}$ which maps 
$\partial D^i = S^{i-1}$ to $X_{i-1}$ and extends
the parameterization of the cell (a map from the interior
of $D^i$ to $X_{i}$).

By Corollary~\ref{TalaskaCor}, if we are
given a perfectly orientable plane-bipartite graph $G$, the image of the parameterization 
$\Meas_G$ of the cell $\Delta_G$ under the Pl\"ucker embedding can be described 
as a map $(t_1, \dots , t_n) \mapsto [h_1(t_1,\dots,t_n), \dots , h_N(t_1,\dots,t_n)]$ 
to projective space, 
where the $h_i$'s are Laurent polynomials with positive coefficients.  
By Lemma \ref{important} and Remark \ref{usefulforCW}, 
the map $\Meas_G$ gives rise to a rational map 
$m_G: X_{P(G)} \to Gr_{k,n}$ which is 
well-defined on $(X_{P(G)})_{\geq 0}$ (a closed ball).  
Furthermore, it is clear that 
the image of $m_G$ on $(X_{P(G)}){\geq 0}$ lies in $(Gr_{k,n})_{\geq 0}$.

Since the totally positive part of the toric variety $X_{P(G)}$ is dense in the 
non-negative part, and the interior gets mapped to the cell 
$\Delta_G$, it follows that $(X_{P(G)})_{\geq 0}$ gets mapped to the 
closure of $\Delta_G$.  Furthermore, 
by construction, $(X_{P(G)})_{>0}$ maps homeomorphically to the 
cell $\Delta_G$.

And now by Lemma \ref{Closure}, it follows
that the boundary of $(X_{P(G)})_{\geq 0}$ gets mapped to the 
$(i-1)$-skeleton of $(Gr_{k,n})_{\geq 0}$.
This completes the proof that the cell decomposition of 
$(Gr_{k,n})_{\geq 0}$ is a CW complex.
\end{proof}

It has been conjectured that the cell decomposition of 
$(Gr_{k,n})_{\geq 0}$ is a regular CW complex which is homeomorphic to a ball.
In particular, if a CW complex is regular then it follows
that the Euler characteristic
of the closure of each cell is $1$.

In \cite{Williams2}, the third author proved that 
the poset of cells of $(\mathbb{G}/P)_{\geq 0}$ is thin and lexicographically shellable, 
hence in particular, {\it Eulerian}.  In other words,
the {\it Mobius function} of 
the poset of cells  takes values $\mu(\hat{0},x) = (-1)^{\rho(x)}$
for any $x$ in the poset.  As the Euler characteristic of 
a finite CW complex is defined to be the number of even-dimensional
cells minus the number of odd-dimensional cells, 
we obtain the following result.

\begin{corollary}
The Euler characteristic of the closure of each cell of 
$(Gr_{k,n})_{\geq 0}$ is $1$.
\end{corollary}


 \section{The face lattice of $P(G)$} \label{faces}
 
 We now consider the lattice of faces of $P(G)$, and 
 give a description in terms of 
 unions of matchings of $G$.  This description is very similar to the description
 of the face lattice of the Birkhoff polytopes, as described by Billera and 
 Sarangarajan \cite{Billera}.  In
 fact our proofs are very similar to those in \cite{Billera}; we just need
 to adapt the proofs of Billera and Sarangarajan
 to the setting of plane-bipartite graphs.

We begin by giving an inequality description of the polytope $P(G)$.
\begin{proposition} \label{Inequalities}
For any plane bipartite graph $G$, the polytope $P(G)$ is 
given by the following inequalities and equations:
$x_e \geq 0$ for all edges $e$, and 
$\sum_{e \ni v} x_e=1$ for each internal vertex $v$. 
If every edge of $G$ is used in some almost perfect matching, 
then the affine linear space defined by the above equations
is the affine linear space spanned by $P(G)$.
\end{proposition}

\begin{proof}
Let $Q$ be the polytope defined by these inequalities. Clearly, $P(G)$ is contained in $Q$. Note that $Q$ lies in the cube $[0,1]^{E(G)}$ because if $e$ is any edge of $G$ and $v$ an endpoint of $e$ then everywhere on $Q$ we have $x_e = 1-\sum_{e' \ni v,\ e' \neq e} x_e \leq 1$.
Let $u$ be a vertex of $Q$. We want to show that $u$ is a $(0-1)$-vector. 
Suppose for the sake of contradiction that $u$ is not a $(0-1)$-vector; let $H$ be the subgraph of
$G$ consisting of edges $e$ for which $0 < u_e < 1$. Note that, if $v$ is a vertex of $H$, then $v$ has degree
at least $2$ in $H$ since $\sum_{e \ni v} u_e=1$. Therefore, $H$ contains a cycle or a path from one boundary vertex of $G$ to another. We consider the case where $H$ contains a cycle, the other case is similar. Let $e_1$, $e_2$, \ldots, $e_{2r}$ be the edges of this cycle; the length of the cycle is even because $G$ is bipartite. 
Define the vector $w$ by $w_{e_i}=(-1)^{i}$ and $w_e=0$ if $e \not \in \{ e_1, e_2, \ldots, e_{2n} \}$.  Let $\epsilon=\min_{i} (\min(u_{e_i}, u_{1-e_i}))$. Then $u+\epsilon w$ and $u-\epsilon w$ are both in $Q$, contradicting that $u$ was assumed to be a vertex of $Q$.

Now, assume that every edge of $G$ is used in some almost perfect matching. Then $P(G)$ meets the interior of the orthant $(\R_{\geq 0})^{E(G)}$, so the affine linear space spanned by $P(G)$ is the same as the affine linear space which cuts it out of this orthant.
\end{proof}

\begin{corollary} \label{DimP}
Suppose that every edge of $G$ is used in some almost perfect matching. Then $P(G)$ has dimension $\# \Faces(G)-1$.
\end{corollary}

\begin{proof}
By proposition~\ref{Inequalities}, the affine linear space spanned by $P(G)$ is parallel to the vector space cut out by the equations  $\sum_{e \ni v} x_e=0$. This is precisely $H_1(G, \partial G)$, where $\partial G$ is the set of boundary vertices of $G$. Let $\tilde{G}$ be the graph formed from $G$ by identifying the vertices of $\partial G$. We embed $\tilde{G}$ in a sphere by contracting the boundary of  the disc in which $G$ lives to a point. Then $H_1(G, \partial G) \cong H_1(\tilde{G})$, which has dimension $\# \Faces(\tilde{G})-1=\# \Faces(G)-1$.
\end{proof}


Note that Corollary~\ref{DimP} is correct even when some components of $G$ are not connected to the boundary, in which case some of the faces of $G$ are not discs.

\subsection{The lattice of elementary subgraphs}

Following \cite{Matching}, we call a subgraph $H$ of $G$ 
{\it elementary} if it contains every vertex of $G$ and 
if every edge of $H$ is used in some almost perfect matching of $H$.
Equivalently, the edges of $H$ are obtained by taking
a union of several almost perfect matchings of $G$.
(To see the equivalence, if $\Edges(H) = \bigcup M_i$, then each edge of $H$ occurs in some $M_i$, which is an almost perfect matching of $H$. Conversely, if $H$ is elementary, then let $M_1$, $M_2$, \dots, $M_r$ be the almost perfect matchings of $G$ contained in $H$ then, by the definition of ``elementary", $\Edges(H) = \bigcup M_i$.)

The main result of this section is the following.

\begin{theorem}\label{facelattice}
The face lattice of $P(G)$ is isomorphic to the lattice of all elementary subgraphs
of $G$, ordered by inclusion.
\end{theorem}

\begin{proof}
We give the following maps between faces of $P(G)$ and elementary subgraphs. If $F$ is a face of $P(G)$, let $K(F)$ be the set of edges $e$ of 
$G$ such that $x_e$ is not identically zero on $F$, and let $\gamma(F)$ be the subgraph of $G$ with edge set $K(F)$. Since $F$ is a face of a $(0-1)$-polytope,
$F$ is the convex hull of the characteristic vectors of some set of matchings, and $\gamma(F)$ is the union of these matchings. Thus, $F \mapsto \gamma(F)$ is a map from faces of $P(G)$ to 
elementary subgraphs. Conversely, if $H$ is a subgraph of $G$, let $\phi(H) = P(G) \cap \bigcap_{e \not \in H} \{ x_e =0 \}$. Since $\{ x_e =0 \}$ defines a face of $P(G)$, the intersection $\phi(H)$ is a face of $P(G)$. From the description in Proposition~\ref{Inequalities}, every face of $P(G)$ is of the form $\phi(H)$ for some subgraph $H$ of $G$. Note also that $\phi(H)=P(H)$.

We need to show that these constructions give mutually inverse bijections between the faces of $P(G)$ and the elementary subgraphs. For any face $F$ of $P(G)$, it is clear that $\phi(\gamma(F)) \supseteq F$. Suppose for the sake of contradiction that $F \neq \phi(\gamma(F))$. Then $F$ is contained in some proper face of $\phi(\gamma(F))$; let this proper face be $\phi(H)$ for some $H \subsetneq \gamma(F)$. Then there is an edge $e$ of $\gamma(F)$ which is not in $H$. By the condition that $e$ is in $\gamma(F)$, the function $x_e$ cannot be zero on $F$, so $F$ is not contained in $\phi(H)$ after all. We deduce that $F=\phi(\gamma(F))$.

Conversely, let $H$ be an elementary subgraph of $G$. It is clear that $\gamma(\phi(H)) \subseteq H$. Suppose for the sake of contradiction that there is an edge $e$ of $H$ which
is not in $\gamma(\phi(H))$. Since $H$ is elementary, there is a matching $M$ of $H$ which contains the edge $e$. Let $\chi_M$ be the corresponding vertex of $\phi(H)$. Then $x_e$ is not zero on $\phi(H)$, so $e$ is in $\gamma(\phi(H))$ after all and we conclude that $H=\gamma(\phi(H))$.
\end{proof}

 The minimal nonempty elementary subgraphs of $G$ are the matchings, corresponding
 to vertices of $P(G)$.  
 
 \begin{corollary}
 Consider a cell $\Delta_G$ of $(Gr_{k,n})_{\geq 0}$ parameterized by 
 a plane-bipartite graph $G$.  For any cell $\Delta_H$ in the closure of $\Delta_G$,
 the corresponding polytope $P(H)$ is a face of $P(G)$.
 \end{corollary}
 
 \begin{proof}
 By \cite[Theorem 18.3]{Postnikov}, every cell in the closure of $\Delta_G$
 can be parameterized using a plane-bipartite graph $H$ which is obtained by deleting
 some edges from $G$.  $H$ is perfectly orientable and hence is an elementary
 subgraph of $G$.  Therefore by Theorem \ref{facelattice}, the polytope
 $P(H)$ is a face of $P(G)$.
 \end{proof}

 \subsection{Facets and further combinatorial structure of $P(G)$}

 We now give a description of the facets of $P(G)$.   
 Let us say that two edges $e$ and $e'$ of $G$ are {\it equivalent} 
 if they separate the same pair of (distinct)
 faces $f$ and $f'$ with the same orientation.
 That is, if we travel across $e$ from face $f$ to $f'$, the black vertex of $e$
 will be to our left if and only if when we travel across $e'$ from $f$ to $f'$,
 the black vertex of $e'$ is to our left.  
 
 \begin{lemma} \label{restrictions}
 If every edge of $G$ is used in an almost perfect matching then two edges $e$ and $e'$ are equivalent if and only if the linear functionals $x_e$ and $x_{e'}$ have the same restriction to $P(G)$.
 \end{lemma}
 
 \begin{proof}
By Proposition~\ref{Inequalities}, the affine linear space spanned by $P(G)$ is cut out by the equations  $\sum_{e \ni v} x_e=1$, where $v$ runs through the internal vertices of $G$. Let $L$ be the linear space cut out by the equations $\sum_{e \ni v} x_e=0$; the polytope $P(G)$ is parallel to $L$ and thus the functionals $x_e$ and $x_{e'}$ have the same restriction to $P(G)$ if and only if they have the same restriction to $L$. In the proof of Corollary~\ref{DimP} we identified $L$ with $H_1(G, \partial G)$. So we just want to determine when the restrictions of $x_e$ and $x_{e'}$ to $H_1(G, \partial G)$ are the same. 

The restrictions of $x_e$ and $x_{e'}$ to $H_1(G, \partial G)$ are elements of the dual vector space $H^1(G, \partial G)$. We can identify $H^1(G, \partial G)$ with the vector space of functions on $\Faces(G)$ summing to zero as follows: Map $\R^{E(G)}$ to $\R^{\Faces(G)}$ by sending an edge $e$ to the function which is $1$ on one of the faces it borders and $-1$ on the other; the sign convention is that the sign is positive or negative according to whether $F$ lies to the right or left of $e$, when $e$ is oriented from black to white. Then $H^1(G, \partial G)$, which is defined as a quotient of $\R^{E(G)}$, is the image of this map.

We now see that $x_e$ and $x_{e'}$ restrict to the same functional on $L$ if and only if they correspond to the same function on the faces of $G$. This occurs if and only if they separate the same pair of faces with the same orientation.
 \end{proof}

 \begin{theorem}
Suppose that $G$ is elementary. Then the facets of 
 $P(G)$ correspond to the elementary subgraphs of the form $G \setminus E$, where $E$ is an equivalence class as above.
  \end{theorem}
 
 \begin{proof}
First, note that if $e$ and $e'$ are not equivalent then, by Lemma~\ref{restrictions}, $x_e$ and $x_{e'}$ have different restrictions to $P(G)$. Thus, there is no facet of $P(G)$ on which they both vanish. On the other hand, if $e$ and $e'$ are equivalent then, again by Lemma~\ref{restrictions}, on every facet of $P(G)$ where $x_e$ vanishes, $x_{e'}$ also vanishes. So we see that every facet of $P(G)$ is of the form $\phi(G \setminus E)$, where $E$ is an equivalence class in $E(G)$. (Here $\phi$ is the function introduced in the proof of Theorem~\ref{facelattice}.)

If $\phi(G \setminus E)$ is a facet of $P(G)$ then $G \setminus E$ is elementary, by Theorem~\ref{facelattice}. Conversely, if $G \setminus E$ is elementary then $\phi(G \setminus E)$ is a face of $P(G)$. Since all the edges of $E$ separate the same pair of faces, $G \setminus E$ has one less face than $G$, so $\phi(G \setminus E)$ is a facet of $P(G)$, as desired.
 \end{proof}

 As a special case of the preceding propositions, we get the following.
 
 \begin{remark}\label{edge}
 Let $N$ be a face of $P(G)$ and let $r$ be the number of regions into which the edges
 of $H(N)$
 divide the disk in which $G$ is embedded.  Then $N$ is an edge of $P(G)$ if and only if 
 $r=2$.  
 Equivalently, two vertices $v_{\O_1}$ and $v_{\O_2}$ of $P(G)$ form an edge
 if and only if $\O_2$ can be obtained from $\O_1$ by switching the orientation along
 a self-avoiding path or cycle in $\O_1$.
 \end{remark}
 
 Recall that the {\it Birkhoff polytope} $B_n$ is the convex hull of 
the $n!$ points
 in $\R^{n^2}$
 $\{X(\pi): \pi \in S_n\}$ where $X(\pi)_{ij}$ is equal to $1$ if $\pi(i)=j$ and is equal to $0$
 otherwise.  It is well-known that $B_n$ is an $(n-1)^2$ dimensional polytope, whose 
 face lattice of $B_n$ is isomorphic to the lattice of all elementary subgraphs of the 
 complete bipartite graph $K_{n,n}$ ordered by inclusion \cite{Billera}.
 Our polytopes $P(G)$ can be thought of as analogues 
of the Birkhoff polytope for 
 planar graphs embedded in a disk.

\section{Appendix: numerology of the polytopes $P(G)$}\label{Numerology}

In this section we give some statistics about a few of the polytopes $P(G)$.  
Our computations were made with the help of the software {\tt polymake} \cite{Polymake}.

\begin{figure}[h]
\centering
\includegraphics[height=1.6in]{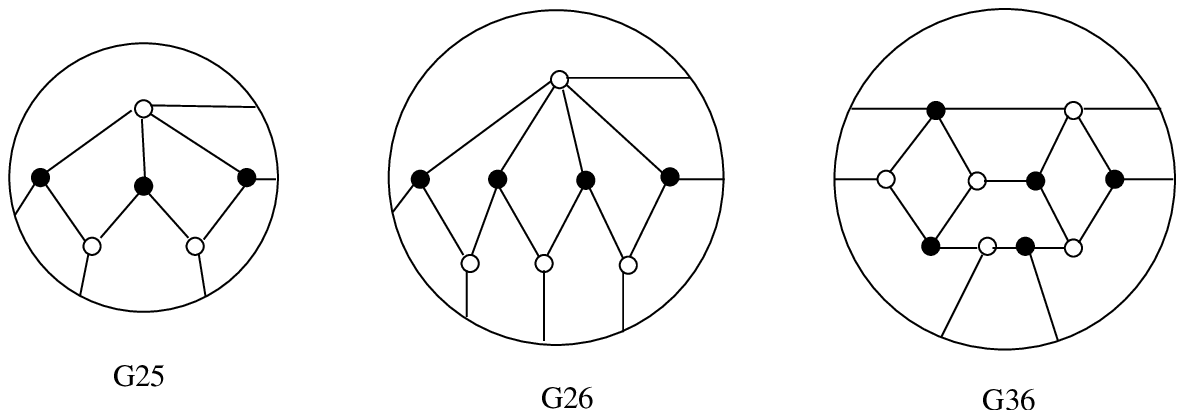}
\caption{}
\label{examples}
\end{figure}

Let $G24$ denote the plane-bipartite graph from Figure \ref{Graph}, and 
let $G25$, $G26$, and $G36$ denote the plane-bipartite graphs shown in Figures \ref{examples}.
These plane-bipartite graphs give parameterizations of the top cells in $(Gr_{2,4})_{\geq 0}, (Gr_{2,5})_{\geq 0}, (Gr_{2,6})_{\geq 0}$, and $(Gr_{3,6})_{\geq 0}$,
respectively.  

The $f$-vectors of the matching polytopes $P(G24)$, $P(G25)$, $P(G26)$ and $P(G36)$ are
\ $(7,17,18,8)$, \ $(14,59,111,106,52,12)$, \ $(25,158,440,664,590,315,98,16)$, \ and \ 
$(42,353,1212,2207,2368,1557,627,149,19)$ respectively.
The Ehrhart series for $P(G24)$, $P(G25)$ and $P(G26)$, 
which give the Hilbert series of the corresponding toric varieties, are 
$\frac{1+2t+t^2}{(1-t)^5}$, $\frac{1+7t+12t^2+4t^3}{(1-t)^7}$, 
and $\frac{1+16t+64t^2+68t^315t^4}{(1-t)^9}$.
The volumes of the four polytopes are $\frac{1}{6}= \frac{4}{4!}$, $\frac{1}{30}=\frac{24}{6!}$, $\frac{41}{10080}=\frac{164}{8!}$, and $\frac{781}{181440}=\frac{1562}{9!}$. Thus, the degrees of the corresponding toric varieties are $4$, $24$, $164$, and $1562$.

\begin{proposition}
Let $G2n$ (for $n \geq 4$) be the family of graphs that extend the 
first two graphs shown in Figure \ref{examples}.  Then
the number of vertices of $G2n$ is given by 
$f_0(G2n)= {n \choose 3} + n-1$.  
\end{proposition}

\begin{proof}
This can be proved by induction on $n$ by removing the leftmost
black vertex.  We leave this as an exercise for the reader.
\end{proof}

Note that in general there is more than one plane-bipartite graph giving a parameterization of a given cell.
But even if two plane-bipartite graphs $G$ and $G'$ correspond to the same cell, in general
we have $P(G) \neq P(G')$.  For example, the plane-bipartite graph in Figure \ref{Alternate} gives a parameterization
of the top cell of $(Gr_{2,6})_{\geq 0}$.  Let us refer to this graph as $\hat{G}26$.  However, 
$P(\hat{G}26) \neq P(G26)$:  the $f$-vector of $P(\hat{G}26)$ is 
$(26, 165, 460, 694, 615, 326, 100, 16)$.

\begin{figure}[h]
\centering
\includegraphics[height=1in]{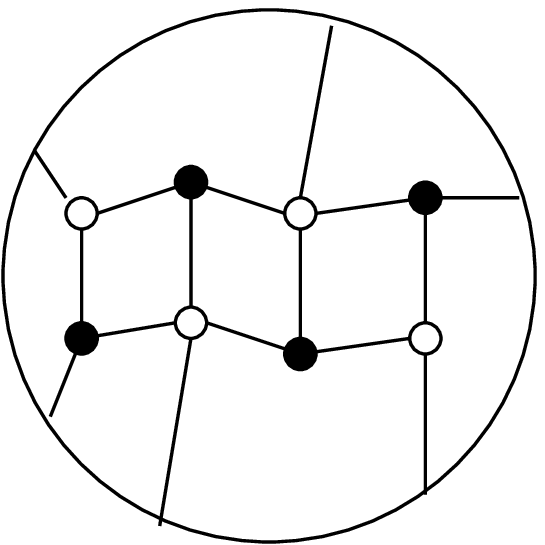}
\caption{}
\label{Alternate}
\end{figure}


\end{document}